\documentclass[12pt]{amsart}

\usepackage{times,epsf,amssymb,amsmath, color,qtree,tikz, wrapfig,mathtools}

\usepackage[hidelinks]{hyperref}
\usepackage{xcolor}
\usepackage{geometry}
\usepackage{tikz-qtree}
\usepackage[capitalise]{cleveref}
\usepackage{subcaption}
\usetikzlibrary{trees} % this is to allow the fork right path
\usepackage{graphicx,xspace,enumitem}       % include graphics

\usepackage{caption}
\newtheorem{thm}{Theorem}[section]
\newtheorem{lem}[thm]{Lemma}
\newtheorem{cor}[thm]{Corollary}

\newtheorem{question}[thm]{Question}

\theoremstyle{definition}
\newtheorem{defn}[thm]{\bf{Definition}}

\theoremstyle{remark}

\newtheorem{rem}[thm]{Remark}

\def\square{\hfill${\vcenter{\vbox{\hrule height.4pt \hbox{\vrule width.4pt height7pt \kern7pt \vrule width.4pt} \hrule height.4pt}}}$}

\newenvironment{pf}{{\it Proof:}\quad}{\square \vskip 12pt}

\title[Minimal Surfaces in Hyperbolic $3$-manifolds]{Existence of Minimal Surfaces\\ in Infinite Volume Hyperbolic $3$-manifolds}

\author{Baris Coskunuzer}
\address[B.~C.]{UT Dallas, Dept. Math. Sciences, Richardson, TX 75080}
\email{coskunuz@utdallas.edu}
\author{Zheng Huang}
\address[Z. ~H.]{Department of Mathematics, The City University of New York, Staten Island, NY 10314, USA}

\email{zheng.huang@csi.cuny.edu}
\author{Ben Lowe}
\address[B.~L.]{Department of Mathematics, University of Chicago, Chicago, IL 60637, USA }
\email{loweb24@uchicago.edu}
\author{ Franco Vargas Pallete}
\address[F.~V-P]{IMPA - Instituto Nacional de Matem\'atica Pura e Aplicada, Rio de Janeiro, RJ, Brasil, 22460-320}
\email{vargas.pallete@impa.br}

\newif\ifshowrevisions
\showrevisionsfalse % turn off color (final version)

\newcommand{\rev}[1]{%
  \ifshowrevisions
    \textcolor{blue}{#1}%
  \else
    #1%
  \fi
}

\newcommand{\PI}{\partial_{\infty}}
\newcommand{\Si}{S^2_{\infty}({\mathbf H}^3)}

\newcommand{\BH}{\mathbf H}

\newcommand{\R}{\mathbf R}
\newcommand{\BZ}{\mathbb Z}

\newcommand{\wh}{\widehat}

\newcommand{\A}{\mathcal{A}}
\newcommand{\C}{\mathcal{C}}

\newcommand{\B}{\mathbf{B}}

\newcommand{\K}{\mathcal{K}}
\newcommand{\U}{\mathcal{U}}

\newcommand{\M}{\mathcal{M}}

\newcommand{\V}{\mathcal{V}}

\newcommand{\G}{\mathcal{G}}

\newcommand{\E}{\mathcal{E}}

\newcommand{\F}{\mathcal{F}}
\newcommand{\h}{\mathcal{H}}
\newcommand{\s}{\mathcal{S}}

\newcommand{\T}{\mathcal{T}}

\newcommand{\cc}{\mathfrak{C}}

\newcommand{\e}{\epsilon}

\newcommand{\dd}{doubly degenerate \xspace}

\newcommand{\hmp}{harmonic map \xspace}

\newcommand{\htm}{hyperbolic $3$-manifold\xspace}

\newcommand{\ms}{minimal surface}
\newcommand{\qf}{quasi-Fuchsian \xspace}
\newcommand{\qfm}{quasi-Fuchsian manifold}

\newcommand{\sff}{second fundamental form \xspace}
\newcommand{\TS}{Teichm\"{u}ller space \xspace}

\begin{document}

\begin{abstract}
The existence of embedded minimal surfaces in non-compact 3-manifolds remains a largely unresolved and challenging problem in geometry. In this paper, we address several open cases regarding the existence of finite-area, embedded, complete, minimal surfaces in infinite-volume hyperbolic 3-manifolds.

Among other results, for doubly degenerate manifolds with bounded geometry, we prove an alternative: either every such manifold contains a closed minimal surface or there exists such a manifold admitting a foliation by closed minimal surfaces. We also construct the first examples of Schottky manifolds with closed minimal surfaces and demonstrate the existence of Schottky manifolds containing infinitely many closed minimal surfaces. Lastly,  for hyperbolic 3-manifolds with rank-1 cusps, we show that a broad class of these manifolds must contain a finite-area, embedded, complete minimal surface.

\end{abstract}
\maketitle

\tableofcontents

\section{Introduction and Statement of Main Results}

The existence of minimal surfaces in Riemannian $3$-manifolds is one of the classical problems in geometric analysis. 
Significant progress has been made in understanding the existence of closed embedded minimal surfaces in compact 3-manifolds, particularly in the positive curvature case. Using the Min-Max approach, Pitts and Simon-Smith established that every closed Riemannian 3-manifold contains a smooth, embedded, closed {\ms}~\cite{pitts2014existence}, \cite{smith1983existence}, \cite{colding2003min}. By the methods of geometric measure theory, Federer \cite{federer2014geometric} demonstrated that for any closed Riemannian 3-manifold, an area-minimizing surface exists in each homology class. Meeks, Simon, and Yau \cite{meeks1982embedded} further showed that every isotopy class contains a minimal surface in such manifolds. More recently, Song \cite{song2018existence} resolved Yau's Conjecture, proving that any closed 3-manifold contains infinitely many minimal surfaces, building on techniques developed by Marques and Neves~\cite{marques2017existence}. 

Various results have been established for compact Riemannian 3-manifolds, but the question of whether embedded closed minimal surfaces exist in non-compact Riemannian 3-manifolds remains open in many contexts. In the particularly important case of infinite-volume hyperbolic 3-manifolds, this problem sits at the intersection of geometric topology and geometric analysis, where the manifold's topology critically influences the existence of closed (or finite-area) minimal surfaces. Recently, the first author~\cite{coskunuzer2020minimal} demonstrated the existence of closed minimal surfaces in several instances of infinite-volume hyperbolic 3-manifolds. However, many key questions remain unresolved. In this paper, we address these open cases by proving multiple zero-infinity dichotomy results and resolving some highly technical rank-1 cases. 

In particular, the first author showed the existence of closed embedded minimal surfaces in all infinite-volume hyperbolic 3-manifolds (\Cref{lem:existence}), with the exception of the following three cases:

\begin{enumerate}
    \item $\M$ is a geometrically infinite product manifold.
    \item $\M$ is topologically a handlebody.
    \item $\M$ contains rank-1 cusps.
\end{enumerate}

Before stating our results, we note that \textbf{throughout the paper, we only consider complete, orientable hyperbolic $3$-manifolds with finitely generated fundamental groups.}
\rev{We recall that a complete hyperbolic $3$-manifold $\M$ homeomorphic to 
$S \times \mathbb{R}$ is called \textit{doubly degenerate} if both of its ends 
are geometrically infinite. A closed, embedded, $2$-sided minimal surface is 
called a \textit{saddle point} minimal surface if it maximizes area in a local 
foliation around it; equivalently, for bumpy metrics, this coincides with being 
an unstable minimal surface. See \Cref{lem:dichotomy2} and \Cref{sec:product} 
for more details.}

Before addressing the missing cases above, we first establish an important preliminary result on the number of minimal surfaces in non-cusped hyperbolic $3$-manifolds by utilizing Song's dichotomy (\Cref{lem:dichotomy2}) for manifolds that are thick at infinity. (See also \cite{stryker2023localization} for results in a similar vein.)

\begin{thm} \label{thm:intro-infinite} Let $\M$ be a complete hyperbolic $3$-manifold with no cusps (i.e., $M= \Gamma \backslash \mathbb{H}^3$ for $\Gamma$ a torsion-free Kleinian group with no parabolic elements.) If $\M$ contains a closed, embedded, saddle point minimal surface, then it contains infinitely many closed, embedded, saddle point minimal surfaces.    
\end{thm}

Next, we study the missing cases above. For case (1) - product manifolds, we prove that an alternative holds when $\M$ has \textit{bounded geometry}, or a uniform lower bound for its injectivity radius at every point:

\begin{thm} [Alternative for Doubly Degenerate Manifolds] \label{thm:intro-product}
Let $\Lambda_\s$ be the space of doubly degenerate hyperbolic $3$-manifolds $\M=\s\times \R$ with bounded geometry. Then, either

\smallskip

\noindent $\bullet$ $\forall \M\in \Lambda_\s$, $\M$ contains a closed, embedded minimal surface in the isotopy class of $\s$, or

\smallskip

\noindent $\bullet$ $\exists \M_0\in \Lambda_\s$ such that $\M_0$ is foliated by closed minimal surfaces $\{\s_t\}$ isotopic to $\s$.

\end{thm}

This theorem establishes a connection between the existence theory of minimal surfaces and the longstanding Hass-Thurston conjecture, on the non-existence of minimal foliations of hyperbolic 3-manifolds (see Question \ref{question:foliation} below). \iffalse We note that the alternatives may not be mutually exclusive.\fi  We note that the two possibilities in the alternative are not mutually exclusive.

Building on the ideas of the proof of this result, we show that the growth rate of disjoint stable minimal surfaces is controlled by the volume of the convex core, conditional on a mild strengthening of the Hass-Thurston conjecture. We are able to prove an unconditional statement for certain quasi-Fuchsian manifolds with short geodesics (\Cref{cor:nonconditional}), and we provide an explicit construction of such manifolds.

Next, we establish a strong dichotomy for the existence of minimal surfaces in quasi-Fuchsian manifolds.

\begin{thm}  [Dichotomy for Quasi-Fuchsian Manifolds]  \label{thm:intro-qf}  A quasi-Fuchsian manifold (without rank-1 cusps) satisfies exactly one of the following statements:
    \begin{itemize}
        \item It has a unique closed minimal surface, which is the area minimizer; or
        \item It contains infinitely many saddle point minimal surfaces.
    \end{itemize}
\end{thm}

For case (2) - hyperbolic handlebodies, we construct the first examples of Schottky manifolds containing closed minimal surfaces. Furthermore, we show existence of Schottky containing infinitely many closed, embedded minimal surfaces.

\begin{thm} \label{thm:intro-Schottky} For any $3$-dimensional handlebody with at least two $1$-handles there exists a homeomorphic Schottky manifold containing infinitely many closed, embedded, saddle point minimal surfaces.   
\end{thm}

Although the infinite family of minimal surfaces given by this result are saddle point minimal surfaces, we also construct examples of Schottky manifolds that contain arbitrarily many stable closed minimal surfaces (\Cref{sec:Schottky}).

 Finally, we show the existence of a complete, finite-area, embedded minimal surface for a broad class of hyperbolic 3-manifolds with rank-1 cusps.  In contrast to the first author's previous work on the existence problem for minimal surfaces using shrinkwrapping surfaces \cite{coskunuzer2020minimal}, our proof uses a doubling trick that depends on an analysis of the behavior of the convex core near the rank one cusps. 

\iffalse we use shrinkwrapping surfaces in the ends \cite{calegari2006shrinkwrapping} to show the existence of a complete, finite-area, embedded minimal surface for a broad class of such manifolds. While our results are more general (see \iffalse \Cref{sec:rank1existence})\fi \ref{thm:existence_1+geofiniteend}, here we highlight an important class of manifolds that this result covers.\fi

\begin{thm} [Manifolds with Rank-1 Cusps] \label{thm:intro-rank1} Let $\M$ be a complete infinite volume hyperbolic $3$-manifold with rank-1 cusps homeomorphic to the interior of a compact manifold with boundary. Assume $\M$ is topologically neither a handlebody nor a geometrically infinite product manifold. If one of the ends of $\M$ is geometrically finite, then $\M$ contains a finite area, complete, embedded, minimal surface $\s$.
\end{thm}

\begin{rem}
A compact, oriented, irreducible, atoroidal 3-manifold with non-empty boundary carries many infinite volume hyperbolic  metrics that have a geometrically finite end with a rank-1 cusp, see e.g. \cite{canary2003approximation}.  
\end{rem}

\subsubsection*{Organization of the paper} In \Cref{sec:background}, we overview the related results and give a basic background on the geometry and topology of hyperbolic $3$-manifolds. 
In \Cref{sec:dichotomy}, we prove the thickness at infinity result for non-cusped hyperbolic $3$-manifolds, and prove~\Cref{thm:intro-infinite}. In~\Cref{sec:product}, we first establish our alternative result for doubly degenerate manifolds (\Cref{thm:intro-product}). Next, we prove our dichotomy result for quasi-Fuchsian manifolds (\Cref{thm:intro-qf}).
In~\Cref{sec:Schottky}, we give the construction of the first examples of Schottky manifolds with closed minimal surfaces~(\Cref{thm:intro-Schottky}). 
In~\Cref{sec:rank1cusps}, we study hyperbolic $3$-manifolds with rank-1 cusps and prove \Cref{thm:intro-rank1}. 
In~\Cref{sec:remarks}, we give our final remarks and discuss the remaining open questions.

\subsubsection*{Acknowledgements} The authors express their gratitude to Antoine Song, Yair Minsky, and Saul Schleimer for generously sharing their insights. We thank Andrea Seppi for helpful comments on Section 4. We thank the referee for carefully reading our paper and for helpful comments. The research of BC was partially supported by the National Science Foundation (Grant No. DMS-2202584) and the Simons Foundation (Grant No. 579977). ZH is partially supported by a PSC-CUNY grant. BL was supported by NSF grant DMS-2202830. FVP was funded by the European Union (ERC, RaConTeich, 101116694). \footnote{Views and opinions expressed are however those of the author(s) only and do not necessarily reflect those of the European Union or the European Research Council Executive Agency. Neither the European Union nor the granting authority can be held responsible for them.}

\section{Background} \label{sec:background}

In this section, we introduce the basics of infinite volume hyperbolic $3$-manifolds and discuss the existence of minimal surfaces within them. \textbf{Throughout the paper, we only consider complete, orientable hyperbolic $3$-manifolds with finitely generated fundamental groups. Similarly, all surfaces discussed are orientable}. For further details on the topic, see ~\cite{minsky2002end,calegari2007hyperbolic,coskunuzer2020minimal}.

\subsection{Finite Volume Hyperbolic $3$-Manifolds} \label{sec:finite_vol}
There exist two types of finite volume hyperbolic $3$-manifolds, classified by whether or not they are compact: compact (closed) and non-compact (cusped). The existence of closed embedded minimal surfaces in closed hyperbolic $3$-manifolds was established decades ago via Almgren-Pitts min-max theory \cite{pitts2014existence}. Since finite volume cusped hyperbolic $3$-manifolds only have rank-2 cusps~\cite{benedetti1992lectures}, this implies that the manifold is homeomorphic to the interior of a compact $3$-manifold where every boundary component is a 2-torus. Recently, the existence of closed minimal surfaces has been shown in this non-compact case by several authors \cite{collin2017minimal, huang2017closed, Rub87}. Furthermore, Chambers-Liokumovich obtained a very general existence result for any finite volume noncompact Riemannian $3$-manifold \cite{chambers2020existence} by generalizing Almgren-Pitts min-max techniques. Recently, Song showed that finite volume $3$-manifolds contain infinitely many closed embedded minimal surfaces \cite{song2023dichotomy}.

\subsection{Rank-1 and Rank-2 Cusps} \label{sec:cusps}

While rank-2 cusps are relatively straightforward, understanding rank-1 cusps in hyperbolic 3-manifolds can be more challenging. Rank-2 cusps have finite volume and are structured as $\mathbb{T}^2 \times [0, \infty)$, whereas rank-1 cusps have infinite volume and take the structure $S^1 \times \mathbb{R} \times [0, \infty)$. Unlike rank-2 cusps, which are isolated within the manifold and do not interact with other ends or cusps, rank-1 cusps behave quite differently. In the next section, we will explain how rank-1 cusps are situated within the manifold, how they are represented by annuli in the boundary of a (relative) compact core, and how they significantly influence the structure of the ends of the hyperbolic manifold.

From a geometric perspective, rank-$1$ cusps can be illustrated as follows. Consider a horoball $\h_q$ in $\BH^3$ with its boundary $\PI \h_q$ at a point $q \in \Si$. Let $\varphi$ be a parabolic isometry that fixes the point $q$. The infinite cyclic group $\G_\varphi$ generated by $\varphi$ is isomorphic to $\BZ$. A rank-1 cusp $\U_k$ is then isometric to the quotient $\h_q/\G_\varphi$ for some $\h_q$ and $\varphi$~\cite{benedetti1992lectures}. For instance, in the upper half-space model of $\BH^3$, let $\h = \{(x, y, z) \mid z \geq 1\}$ and $\varphi(x, y, z) = (x + 1, y, z)$. For these $\h$ and $\varphi$, a rank-1 cusp $\U$ is isometric to $[0,1] \times \R \times [1,\infty)$, where the boundaries $\{0\} \times \R \times [1,\infty)$ and $\{1\} \times \R \times [1,\infty)$ are identified, forming $S^1 \times \R \times [1,\infty)$. Similarly, rank-2 cusps are formed by taking the quotient of $\h_q$ by the action of $\BZ \oplus \BZ$, generated by two independent parabolic isometries that fix the point $q \in \Si$. For example, with the same notation, if $\psi(x, y, z) = (x, y + 1, z)$, then a rank-2 cusp $\V$ would be isometric to $\h/\langle \varphi, \psi \rangle$, which corresponds to $[0,1] \times [0,1] \times [1,\infty)$ with opposite sides identified.

\subsection{Infinite Volume Hyperbolic $3$-Manifolds} \label{sec:infinitevolume}
These manifolds are classified based on the geometry of their ends \cite{thurston2022geometry, minsky2002end}. Let $\M$ be a complete hyperbolic $3$-manifold with infinite volume, and let $\C_\M$ denote a \textit{compact core} of $\M$ (see \Cref{fig:end2DCC}). The compact core $\C_\M$ is a compact codimension-$0$ subset of $\M$ such that $\M$ is homeomorphic to the interior of $\C_\M$ \cite{scott1973compact, calegari2007hyperbolic}. The manifold can then be decomposed as $\M - \C_\M = \bigcup_{i=1}^n \E_i$, where each $\E_i$ is referred to as an \textit{end of} $\M$. The topology of each end is $\s_i \times (0, \infty)$, where $\s_i$ is a component of $\partial \C_\M$, a closed surface of genus $\geq 2$ \cite{agol2004tameness, calegari2006shrinkwrapping}. We call these ends $\E_i$ \textit{topological ends} of $\M$. This description applies to any hyperbolic $3$-manifold, regardless of whether it contains rank-1 cusps. If $\M$ has no rank-1 cusps, this is the only decomposition we use.

However, \textit{in the presence of rank-1 cusps}, a subtler decomposition is needed, through \textit{relative compact cores}~\cite{canary2006notes}. If $\M$ is an infinite volume hyperbolic $3$-manifold with cusps, let $\M^0$ be the manifold obtained from $\M$ by removing all rank-$1$ and rank-$2$ cusps. \rev{Here, a \textit{toroidal component} of $\partial \M^0$ is a boundary 
component homeomorphic to a $2$-torus $\mathbb{T}^2$, corresponding to a 
rank-2 cusp, and an \textit{annular component} of $\partial \M^0$ is a boundary 
component homeomorphic to $S^1 \times \mathbb{R}$, corresponding to a rank-1 
cusp. An annulus $A \subset \partial \wh{\C}_\M$ is called \textit{incompressible} 
if the inclusion-induced map $\pi_1(A) \rightarrow \pi_1(\M)$ is injective.} 
A \textit{relative compact core} $\wh{\C}_\M$ for $\M$ is a compact core of 
$\M$ which intersects each toroidal component of $\partial \M^0$ in the entire 
torus, and intersects each annular component of $\partial \M^0$ in a single 
incompressible annulus (See~\Cref{fig:end2DRelCC}). In particular, relative compact core $\wh{\C}_\M$ is a compact core where one can identify all rank-$1$ and rank-$2$ cusps in $\M$ through tori and marked annuli in $\partial \wh{\C}_\M$, respectively.

\begin{figure*}[t]
     \centering
     \begin{subfigure}[b]{0.32\textwidth}
         \centering
         \includegraphics[width=\linewidth]{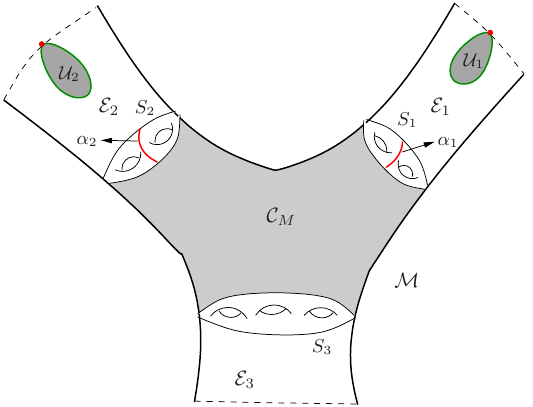}
         \caption{\scriptsize Rank-1 cusps}
         \label{fig:end3D}
     \end{subfigure}
    \begin{subfigure}[b]{0.32\textwidth}
         \centering
         \includegraphics[width=\linewidth]{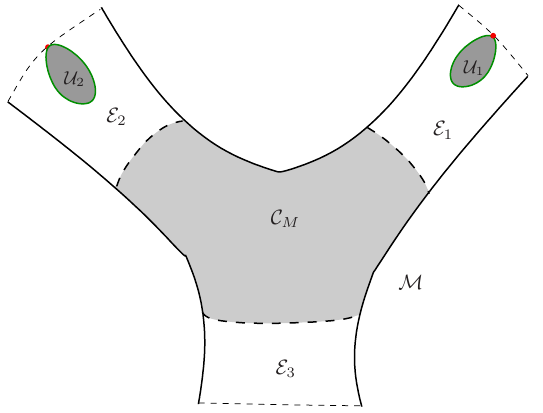}
        \caption{\scriptsize Topological Ends}
         \label{fig:end2DCC}
     \end{subfigure}
     \begin{subfigure}[b]{0.32\textwidth}
         \centering
         \includegraphics[width=\linewidth]{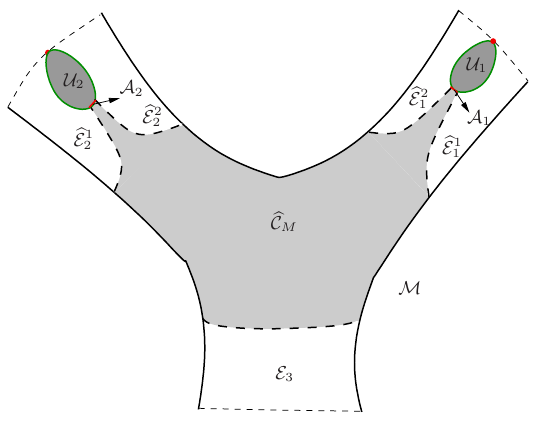}
       \caption{\scriptsize Geometric Ends}
         \label{fig:end2DRelCC}
     \end{subfigure}
        \caption{\scriptsize 2D illustration of the compact core $\C_M$ (center) and the relative compact core $\wh{\C}_M$ (right) for a hyperbolic $3$-manifold $\M$ with rank-1 cusps $\U_i$. The compact core $\C_M$ decomposes $\M$ into its topological ends, $\M - \C_M = \bigcup \E_i$, where each end $\E_i = \s_i \times [0, \infty)$. The rank-1 cusps $\U_i$ correspond to specific curves $\alpha_i$ on $\s_i$ (left). The relative compact core $\wh{\C}_M$ further decomposes $\M$ into geometric ends $\{\wh{\E}^j_i\}$. If $\{\alpha_k\}$ separates $\s_i$, such as $\s_1 - \alpha_1 = \mathcal{T}^1_1 \cup \mathcal{T}^2_1$, then $\M - \wh{\C}_M = \bigcup \wh{\E}^j_i \cup \bigcup \U_k$, where each $\wh{\E}^j_i \simeq \mathcal{T}^j_i \times [0, \infty)$. Here, $\wh{\E}^j_i$ are also referred as cusped ends as $\mathcal{T}^j_i$ are cusped (open) surfaces.
        \label{fig:ends}}
                \vspace{-.1in}
\end{figure*}

We call each component of $\M^0 - \wh{\C}_\M = \bigcup_{i,j} \wh{\E}_i^j$ as a 
\textit{geometric end} of $\M$, where $\wh{\E}_i^j \subset \E_i$ (a topological 
end). By \cite{bonahon1986bouts}, every hyperbolic $3$-manifold $\M$ with a finitely 
generated fundamental group admits a \textit{relative compact core} $\wh{\C}_\M$. \rev{The 
components of $\partial \wh{\C}_\M - \partial \M^0$ are (possibly punctured) 
surfaces $\{\T_i^j\}$, where $i$ runs over the $n$ topological ends of $\M$, 
and for each $i$, $j$ runs from $1$ to $m_i$. Here, $m_i$ is the number of 
connected components obtained by removing from the closed surface $\s_i \subset 
\partial\C_\M$ the core curves of all rank-1 cusps within $\E_i$. When one or 
more of these core curves are separating in $\s_i$, they split $\s_i$ into 
$m_i > 1$ components (see \Cref{fig:ends}). We call each corresponding component 
$\wh{\E}_i^j \simeq \T_i^j \times [0,\infty)$ of $\M^0 - \wh{\C}_\M$ a 
\textit{geometric end} of $\M$, where $\wh{\E}_i^j \subset \E_i$ (a topological 
end). The geometric ends $\{\wh{\E}_i^j\}$ of $\M^0$ thus correspond one-to-one 
with the components $\{\T_i^j\}$ of $\partial \wh{\C}_\M - \partial \M^0$. If 
no rank-1 cusp exists in a given topological end $\E_i$, then $m_i=1$ and the 
topological and geometric ends coincide.}

To clarify, in the presence of rank-1 cusps, the annuli $\{\A_k\}$ corresponding to rank-1 cusps $\{\U_k\}$ in $\partial \wh{\C}_\M\cap \partial \M^0$ are removed from $\partial \wh{\C}_\M$, and some of the surfaces $\T_i^j$ in the decomposition $\bigcup_{i,j} \T^j_i=\partial \wh{\C}_\M - \partial \M^0$ are no longer closed, but a punctured surface. Hence, we also call the ends $\E_i\simeq S_i\times[0,\infty)$ corresponding to such punctured surfaces $\T_i^j$ {\em cusped ends} (See~\Cref{fig:ends}).

For example, let $\Sigma$ be a punctured torus, and consider a hyperbolic structure on $\M=\Sigma\times \R$ where the puncture in $\Sigma$ corresponds to a rank-1 cusp. Then, $\M$ would be topologically nothing but a genus 2 handlebody, which has one topological end. On the other hand, $\M^0=\wh{\Sigma}\times \R$ where $\wh{\Sigma}$ would be the compact surface obtained by removing a small neighborhood of the puncture in $\Sigma$. Then, our relative compact core decomposition would produce two ends $\Sigma^+\cup\Sigma^-=\partial \wh{\C}_\M - \partial \M^0$ where $\partial \M^0=\partial\wh{\Sigma}\times \R$. Therefore, we would have two geometric ends $\E^\pm=\Sigma^\pm\times [0,\infty)$ in $\M$ along with one rank-1 cusp. See~\cite{marden2007deformations} for further details.

The \textit{convex core} $\cc_\M$ of $\M$ is the smallest convex submanifold of $\M$, which is homotopy equivalent to $\M$. 
We call an end $\E$ of $\M^0$ \textit{geometrically finite} if it has a neighborhood that does not intersect the convex core $\cc_\M$. This simply means that the end $\E$ geometrically has a product structure, i.e., $\E=S\times[0,\infty)$ where the areas of the equidistant surfaces $S\times \{t\}$  are exponentially growing with $t$. If an end $\E$ is not geometrically finite, then we call it geometrically infinite. In that case, the end $\E$ is contained in the convex core $\cc_\M$.
For $\wh{\C}_\M$ a relative compact core of $\M$, we call an end $\E$ of $\M^0$ \textit{simply degenerate} if there exists a sequence of geodesics $\{\alpha_i\}\in \E$ leaving every compact subset of $\M$ (escaping to infinity). Bonahon (\cite{bonahon1986bouts}) showed that every geometrically infinite end is simply degenerate: if $\E\simeq S\times [0,\infty)$ for $S$ is a component in $\partial \wh{\C}_\M-\partial \M^0$, then for any $K>0$, there exists an $N$, such that $\alpha_n\in S\times [K,\infty)$ for any $n\geq N$.

Geometrically infinite ends are classified into two categories as follows: If the geometrically infinite end $\E$ has a positive injectivity radius, then we call $\E$ has {\em bounded geometry}. Otherwise, we say that $\E$ has {\em unbounded geometry}. In both cases, there exists a sequence of geodesics $\{\alpha_n\}$ in the end $\E$ escaping to infinity with $|\alpha_n|$ bounded, while for the unbounded geometry case, we also have $|\alpha_n|\to 0$.

A complete hyperbolic $3$-manifold $\M$ is called {\em geometrically finite} if all ends are geometrically finite. This is equivalent to saying that the convex core $\cc_\M$ has finite volume. Otherwise, we call the manifold {\em geometrically infinite}. For more details on the ends of hyperbolic $3$-manifolds, see~\cite{canary1993ends,minsky1994rigidity}.

\subsection{Existence Results} As discussed before, until a decade ago, very little was known about the existence of closed, embedded minimal surfaces in noncompact hyperbolic $3$-manifolds except in cases like geometrically finite manifolds. Finite volume non-compact hyperbolic $3$-manifolds are done with results by the authors mentioned above. By adapting shrinkwrapping techniques from~\cite{calegari2006shrinkwrapping}, the first author proved the existence of closed embedded minimal surfaces in most infinite volume hyperbolic $3$-manifolds.

\begin{lem} \label{lem:existence} \cite{coskunuzer2020minimal} Let $\M$ be an infinite volume complete hyperbolic $3$-manifold. Then, $\M$ contains a closed, embedded minimal surface $S$ if $\M$ is not one of the following three categories:
\begin{enumerate}
    \item $\M$ is a geometrically infinite product manifold.
    \item $\M$ is topologically a handlebody.
    \item $\M$ contains rank-1 cusps.
\end{enumerate}
    
\end{lem}

Although some nonexistence examples have been given for case (2) involving $\BH^3$ or Fuchsian Schottky manifolds in~\cite{coskunuzer2020minimal}, no such results exist for most cases in (1) and (3). In this paper, we address all three cases and present new results.

\section{Dichotomy for Non-Cusped Hyperbolic $3$-manifolds} \label{sec:dichotomy}

In this section, we first prove that hyperbolic $3$-manifolds with no cusps are thick at infinity, i.e., it cannot contain a non-compact, finite area, complete, embedded minimal surface. Then, by using Song's dichotomy result~\cite{song2023dichotomy} for manifolds thick at infinity, we show that any such manifold either contains infinitely many, closed, embedded, unstable minimal surfaces or it contains no unstable, embedded, closed minimal surfaces. First, let us recall a few notions.

\begin{defn} [Manifolds thick at infinity \cite{song2023dichotomy}]
Let $(\M^{n+1},g)$ be a complete $(n+1)$-dimensional Riemannian manifold. It is said to be \textit{thick at infinity} (in the weak sense) if any connected, finite volume, finite index, complete, minimal hypersurface in $(\M,g)$ is closed.
\end{defn}

\begin{rem}
Song did not impose the finite index condition in his definition of thick at infinity, but the results of his paper continue to hold with the above definition of thickness at infinity, that applies only to finite index minimal surfaces. It is possible but not entirely straightforward to verify that, for example, closed Riemannian manifolds are thick at infinity in the original definition (without the finite index condition), but we do not give the argument here.  
\end{rem}

Song proved the following zero-infinity dichotomy for manifolds thick at infinity.

\begin{lem}\cite[Theorem 1.2]{song2023dichotomy}  \label{lem:dichotomy2}
Let $(\M,g)$ be an $(n+1)$-dimensional complete manifold with $2\leq n\leq 6$, thick at infinity. Then the following dichotomy holds true:
\begin{enumerate}
\item either $\M$ contains infinitely many saddle point minimal hypersurfaces,
\item or there is no saddle point minimal hypersurface in $\M$. 
\end{enumerate}
\end{lem}
Here a 2-sided minimal hypersurface is said to be a \emph{saddle point} if it maximizes area in a local foliation around it, whereas a saddle point 1-sided minimal hypersurface is defined as such if the 2-sided double cover is a saddle point in the corresponding double cover of a small tubular neighborhood of the minimal hypersurface. Note that if the metric is bumpy (i.e. no closed minimal hypersurface has a non-trivial Jacobi field) then saddle point minimal hypersurfaces are exactly \textit{unstable 2-sided closed minimal hypersurfaces} and \textit{1-sided closed minimal hypersurfaces with unstable double cover}. For more details, see~\cite[Section 3]{song2023dichotomy}.

Next, we show any hyperbolic $3$-manifold (compact or noncompact) with no cusps is thick at infinity.

\begin{lem} \label{lem:thickness} Any complete hyperbolic $3$-manifold without cusps is thick at infinity.    
\end{lem}

\begin{pf} Let $\M$ be an hyperbolic $3$-manifold with no cusp. We proceed by contradiction. Assume to the contrary that there is a non-compact, complete, finite index, and finite area minimal surface $\s$ in $\M$ exists, where the area $\|\s\|=C$. 

First, we show that $\s$ cannot completely stay in a compact region in $\M$, i.e. $\PI\s\neq \emptyset$. Assume on the contrary that $\s\subset \K_0$ where $\K_0$ is a compact region in $\M$. \iffalse Then we can find a limit point $q\in \K_0$ of $\s$.\fi Fix $\e_0>0$. Let $d_{\s}(\cdot,\cdot)$ represent the intrinsic distance in $\s$. Then, as $\s$ is a complete, embedded surface, there is \iffalse a point $q \in \K_0$ and\fi  a sequence of points $\{p_i\}$ in $\s\cap \B^\M_{\e_0}(q)$ with \iffalse $p_i\to q$ in $\M$, and \fi $d_\s(p_i,p_j)\geq 2\e_0$ for any $i<j$. Let $\B^{\s}_{\e_0}(p_i)$ represent the intrinsic $\e_0$-disk in $\s$ with center $p_i$. Since  $\s$ has finite index, we know that for all but finitely many of the $p_i$ the $\e_0$-disk centered at $p_i$ is stable.  Therefore by the fact that $\M$ has bounded geometry and the curvature estimates for stable minimal surfaces (see for instance \cite{Sch83}), for all but finitely many of the $p_i$ the injectivity radius of $\s$ at $p_i$ is at least some constant $\epsilon< \e_0$.  

Since $\s$ is a minimal surface in a hyperbolic $3$-manifold, the sectional curvature of $\s$ is bounded from above by $-1$. Then, by a simple estimate~\cite[Section 1.1]{calegari2006shrinkwrapping}, we have the area of intrinsic disk $\|\B^\epsilon (p_i)\|\geq 2\pi(\cosh{\epsilon}-1)\geq \pi r^2$. As $d_\s(p_i,p_j)\geq 2\e_0$, $\B^{\s}_{\e_0}(p_i)\cap \B^{\s}_{\e_0}(p_j)=\emptyset$ for any $i<j$. Since $C=\|\s\|\geq \sum_{i=1}^\infty \|\B^\s_{\epsilon }(p_i)\|$, this is a contradiction. This shows $\PI\s\neq \emptyset$. This also finishes the case when $\M$ is a closed hyperbolic $3$-manifold.

Now, assume that $\s$ is not contained in a compact subset of $\M$. If $\M$ has bounded geometry (say injectivity radius $\rho_\M>\e_0>0$), choose a sequence of points $\{p_k\}$ on $\s$ tending to infinity and at a pairwise distance from each other of at least $1$ inside $\M$.  Then since $\M$ has bounded geometry, we can apply the monotonicity formula which implies that the area of the ball of radius $1$ centered at $p_k$ in $\M$ intersected with $\s$ has area at least some definite $\epsilon'>0$. This contradicts the fact that $\s$ has a finite area.  

If $\M$ has unbounded geometry, let $\e_3$ be the Margulis constant. Let $\M_{thin}$ be the thin part of $\M$, i.e., $\M_{thin}=\{x\in\M\mid \rho_\M(x)\leq \e_3\}$ where $\rho_\M(x)$ is the injectivity radius of $\M$ at $x$. Let $\M_{thick}=\M-\M_{thin}$. Since $\M$ has no cusps, $\M_{thin}$ is a \underline{disjoint} union of Margulis tubes~\cite{benedetti1992lectures,futer2019effective}, i.e., $\M_{thin}=\bigsqcup_i\mathbf{T}_{\gamma_i}$ where $\mathbf{T}_\gamma$ is the Margulis tube around the short geodesic $\gamma$. This implies $\s-\M_{thin}$ contains a sequence $\{p_k\}$ diverging to infinity, with $d(p_k,p_j)>2\e_3$. Notice, as $\{p_k\}\in \M_{thick}$, $\rho_\M(p_k)\geq \e_3$ for any $k$. The estimate from \cite{calegari2006shrinkwrapping}[Section 1.1] applies as it did two paragraphs above to show that the area of the ball of radius $\e_3$ centered at $p_k$ in $\M$ intersected with $\s$ has area at least some definite $\e'>0$. Again, this contradicts the fact that $\s$ has finite area. The proof follows.    
\end{pf}

By combining this result with Song's dichotomy, we obtain the following corollary:

\begin{cor} \label{cor:zero-infinity} Let $\M$ be a complete hyperbolic $3$-manifold with no cusps. If $\M$ contains a closed, embedded, saddle point minimal surface, then it contains infinitely many closed, embedded, saddle point minimal surfaces.    
\end{cor}

We end this section with two brief remarks:
\begin{rem} [Stable Minimal Surfaces] We note that the above results and Song's dichotomy theorem shed no light on the number of stable minimal surfaces in a manifold $\M$. A manifold can have a stable minimal surface without closed saddle point minimal surfaces. For example, let $\M\simeq \s\times \R$ be a Fuchsian manifold, then $\M$ contains a unique minimal surface, which is the totally geodesic surface $\s\times\{0\}$. More generally, every almost Fuchsian manifold \cite{uhlenbeck1983closed} only contains exactly one closed {\ms}, and there are also examples that are not almost Fuchsian \cite{huang2021beyond}.     
\end{rem}

\begin{rem} [Thickness for Cusped Hyperbolic $3$-Manifolds] We note that our thickness at infinity result does not generalize to hyperbolic 3-manifolds with cusps. For this case, a straightforward example is a Fuchsian manifold $\F\times \R$ where $\F$ is a punctured surface. We provide further examples of cusped, quasi-Fuchsian hyperbolic $3$-manifolds containing finite-area, non-compact, complete, embedded minimal surfaces in Section \ref{sec:rank1cusps}.
\end{rem}

\section{Product Hyperbolic $3$-Manifolds} \label{sec:product}

In this section, we first present our alternative result for {\dd \ \htm}s with bounded geometry. These manifolds can be seen as limits of {\qfm}s. As an application, we obtain results on counting minimal surfaces in these manifolds. In particular, we derive a linear growth estimate for the number of disjoint embedded {\ms}s in {\qfm}s with short geodesics (\Cref{sec:QFcount}), and we provide an explicit construction of these manifolds (\Cref{sec:shortgeodesics}).

\subsection{Alternative for Doubly Degenerate Manifolds} 

In this subsection, we say $\M$ is {\it\dd} if $\M$ is a complete {\htm} homeomorphic to $S \times \mathbb{R}$, where both ends are geometrically infinite. We say $\M$ has {\it bounded geometry} if the injectivity radius of $\M$ has a positive lower bound, i.e. $\rho_\M \ge \epsilon_0>0$.

Before stating our result, we need the following dichotomy result, which is an application of the implicit function theorem.  

\begin{lem} \label{lem:Anderson} 
Let $S$ be a closed essential locally area-minimizing surface in a complete {\htm} $M$ diffeomorphic to $S \times \mathbb{R}$ of bounded geometry. Then either 
\begin{enumerate} 
\item The manifold $M$ is foliated by {\ms}s isotopic to $S$, or
\item There is a foliation $S_u$ of a tubular neighborhood of $S$, $u \in (-\epsilon,\epsilon)$, $S_0=S$, so that the mean curvature vector of $S_u$ for $u \neq 0$ points  towards $S$.
\end{enumerate}   
\end{lem} 

\begin{proof}
If $S$ is strictly stable, then the images under the normal exponential map of small multiples of the bottom eigenfunction of the second variation operator of $S$ multiplied by a unit normal vector field form a mean-convex foliation of a tubular neighborhood of $S$ (see e.g. the proof of Theorem 5.2 in \cite{huang2021beyond}.)  

If $S$ is degenerate stable, assume first that no neighborhood of $S$ admits a foliation by minimal surfaces. Let $\phi$ be a Jacobi field of the second variation operator, which we can take to be positive and with $L^2$-norm equal to $1$ (recall that the space of such Jacobi fields is one-dimensional, and each nonzero Jacobi field is nowhere vanishing.  This is a standard fact about bottom eigenfunctions of linear elliptic operators \cite{evans2022partial}[Chapter 6].) Let $\nu(x)$ be the unit normal vector to the point $x\in S$, determined by fixed choices of orientation for $S$ and $M$. For $w$ a smooth function of sufficiently small $C^{2,\alpha}$-norm, consider the functional $F(w,t) = H_w(x) - t\phi(x)$.  Here $H_w$ is the real valued function on $S$ defined at $x_0$ by taking the inner product of the mean curvature vector of the graphical surface parametrized by $x \mapsto exp_x(w(x)\nu(x))$ at $exp_{x_0}(w(x_0)\nu(x_0))$, with the unit normal vector to that surface at $exp_{x_0}(w(x_0)\nu(x_0))$ pointing towards the same end of $M$ as the normal vector field $\nu$ to $S$.

  Then $F$ maps a neighborhood of $(0,0)$ in $C^{2,\alpha}(S) \times \mathbb{R}$ to $C^{0,\alpha}(S)$, and $DF(0,0)$ is surjective with one-dimensional kernel given by $\phi$.  Surjectivity follows from diagonalizing the second variation operator restricted to the orthogonal complement of its one-dimensional kernel, which is spanned by $\phi$. Applying the implicit function theorem gives a 1-parameter family of graphical surfaces $S_u$, $S_0=S$, so that the mean curvature of each $S_u$ (viewed as a function on $S$ via pairing the mean curvature vector of each $S_u$ with the unit normal vector pointing in the same direction as $\nu$ and then precomposing with the normal exponential map at $S$) is a scalar multiple $t(u)$ of $\phi(x)$. 

The function $t(u)$ so that the mean curvature of $S_u$ is $t(u) \phi$ is analytic.  Therefore, if $S_u$ had vanishing mean curvature for a sequence of $u$ tending to zero, then $t(u)$ would have to vanish identically (Compare \cite{anderson1983cmh}[Section 5].) Therefore for $u$ small enough, $t(u)=0$ only if $u=0$.  Similarly, $t'(u)$ cannot vanish for a sequence of $u$ tending to zero, and must therefore be nonzero for $u$ small enough and nonzero.  Because $S_0=S$ is locally area-minimizing, $t'(u)$ must be negative for non-zero u sufficiently small, and $t(u)$ is decreasing.  

Because $(S_u, t(u))$ must be tangent to the kernel of $DF(0,0)$ at $u=0$, we know that the derivative in $u$ of the function of which $S_u$ is the graph over $S$ at $u=0$ is a nonzero multiple of $\phi$. Since $\phi$ is nowhere vanishing, this implies that $S_u$ is disjoint from $S$ for $u$ sufficiently small.  

For a fixed small $\epsilon_0>0$, we can find $\epsilon'>0$, $\epsilon'<<\epsilon_0$, so that $S_u$ is contained in the region bounded by $S_{-\epsilon_0}$ and $S_{\epsilon_0}$ for $u \in (-\epsilon', \epsilon')$, and so that $S_u$ is disjoint from $S$ if $u \neq 0$.  

We claim also that $S_{u_1}$ is disjoint from $S_{u_2}$ if $u_1 \neq u_2$ and $u_1,u_2 \in (-\epsilon',\epsilon')$.  If $u_1$ and $u_2$ have different signs then this is clear, so assume that $u_1>u_2>0$ (the case where both are negative follows from the same argument.) Suppose for contradiction that $S_{u_2}$ is not contained in the region bounded by $S_{u_1}$ and $S$. Then $S_{u_1}$ must intersect $S_{u_2}$ because it is disjoint from $S$.  We let $u'$ be the supremum of the set of parameters $u$, $\epsilon_0\geq u>0$, so that $S_u \cap S_{u_2}$ is non-empty.  Then $S_{u'}$ and $S_{u_2}$ must intersect non-transversely so that $S_{u_2}$ is contained on one side of $S_{u'}$.  Since $u'\geq u_1 >u_2$, this is a contradiction by the strong maximum principle and the fact that since $t(u)$ is monotone, the mean curvature of $S_{u'}$ has the same sign but is larger in magnitude than that of $S_{u_2}$.  It follows that, for $u\in (-\epsilon',\epsilon')$, the $S_u$ give a mean-convex foliation of a neighborhood of $S$.

\iffalse 
Work by Anderson \cite{anderson1983cmh}[Section 5] implies that either $S$ has a neighborhood where it is the unique minimal surface, or there is a foliation of a neighborhood of $S$ by minimal surfaces isotopic to $S$, and thus by bounded geometry of $M$ a foliation of all of $M$ by minimal surfaces isotopic to $S$.  If it were not the case that the mean curvature of $S_u$ were a positive multiple of $\phi(x)$ for all $u$ sufficiently close to zero, then we claim that there would be a sequence $S_{u_n}$, $u_n \neq 0$, $u_n \to 0$, so that $S_{u_n}$ was minimal. The work by Anderson would then imply that a tubular neighborhood of $S$, and thus all of $M$ admits a minimal foliation.  To see the claim, if none of the $S_u$ were minimal for small enough $u$, then their mean curvature vectors paired with unit normals pointing towards the same end as $\nu$ would all have to be positive or negative multiples of $\phi$ for $u>0$ or $u<0$.  But the latter is impossible because $S$ is locally area-minimizing.  
\fi

Therefore if a neighborhood of $S$ in $M$ does not admit a foliation by minimal surfaces, it must be the case that there is a neighborhood of $S$ with a mean-convex foliation: i.e., for which all surfaces in the foliation not equal to $S$ have mean curvature vectors pointing towards $S$.  Otherwise if a neighborhood of $S$ has a local minimal foliation, then since $M$ has bounded geometry it also has a global minimal foliation \cite{anderson1983cmh}[Section 5]. %the mean curvature of $S_u$ is a positive multiple of $\phi(x)$ for $u$ in some small interval containing zero,  which shows the existence of a local mean-convex foliation. 

 \end{proof}

\begin{rem} \label{rem:Anderson}
The condition locally area-minimizing is essential here. Using the fact that the function $f(t)=t^3$ vanishes to second order at $0$, it is possible to construct Riemannian manifolds $M$ for which the conclusion of the lemma can fail if $S$ is only assumed to be stable. Furthermore, \iffalse the work in \cite{GLP21} shows that\fi it can also fail for $M$ hyperbolic and satisfying the assumptions of the lemma \cite{GLP21}.  
\end{rem} 

Now, we are ready to state the main results of this section.

\begin{thm} \label{thm:product-dichotomy} Let $\M$ be a {\dd}{\htm} with bounded geometry. Then either $\M$ contains a closed {\ms} diffeomorphic to the fiber $S$, or there exists a {\dd}manifold with bounded geometry which is foliated by closed minimal surfaces.  
\end{thm}

\begin{pf}
Fix a homotopy equivalence $f:S\rightarrow \M$. This defines a functional $E: \text{Teich}(S) \rightarrow \mathbb{R}$ on {\TS}$\text{Teich}(S)$ as follows.  We regard each point $\sigma$ in Teich(S) as a hyperbolic surface $S_\sigma$ marked by a diffeomorphism $\Psi_{\sigma}:S \rightarrow S_{\sigma}$.  Then $E(\sigma)$ is defined to be the energy of the unique harmonic map $S_{\sigma} \rightarrow \M$ in the homotopy class of $f \circ \Psi_{\sigma}^{-1}$. \iffalse for each conformal structure $\sigma \in \text{Teich}(S)$, $E(\sigma)$ is the energy of the unique harmonic map $F_{\sigma}:(S,\sigma) \rightarrow \M$.\fi  It follows from \cite{minsky1992harmonic}[Theorem 3.6] that $F_{\sigma}$ exists and is unique.  We are going to use the standard fact that the image of a harmonic map corresponding to a critical point of $E$ is a minimal surface in $\M$.

If $E$ achieves its infimum then we are done, so assume this is not the case, and choose a minimizing sequence $\sigma_n$ for $E$ in $\text{Teich}(S)$.  We know that the infimal energy is positive by \cite{minsky1992harmonic}.  For each $\sigma_n$, choose a point $p_n$ in $\M$ contained in the image of the corresponding harmonic map into $\M$.  Since $\M$ has bounded geometry, we can pass to a limit of the pointed hyperbolic manifolds $(\M,p_n)$ to obtain a pointed limit manifold $(\M_\infty,p)$ \cite{mcmullen1996renormalization}.  This means that there are compact sets $K_n$ homeomorphic to $S \times [0,1]$ that exhaust  $(\M_\infty,p)$ and maps $\Phi_n:K_n \rightarrow (\M,p_n)$ that smoothly converge to isometries onto their images. Moreover, it follows from the bounded geometry of $\M$ that $\M_\infty$ is also homeomorphic to $S \times \mathbb{R}$.

\iffalse Fix identifications  of $\M$ and $\M_\infty$ with $S \times \mathbb{R}$.\fi  Fix a homotopy equivalence $f_{\infty}:S \rightarrow \M_{\infty}$, and recall that we fixed a homotopy equivalence $f: S \rightarrow \M$ earlier.  Then given these markings, each of the maps $\Phi_n$ gives an element $G_n$ of the mapping class group of $S$. For each point $\sigma$ in the {\TS}of $S$, we obtain a point $G_n(\sigma)$ in $\text{Teich}(S)$ that projects to the same point in moduli space as $\sigma$. Here $G_n(\sigma)$ is the point that $G_n$ maps $\sigma$ to for the action of the mapping class group on Teichmuller space.

For each $\sigma \in \text{Teich}(S)$, denote the corresponding harmonic maps $S \rightarrow \M$ and $S \rightarrow \M_\infty$ by $F_{\sigma}$ and $F_{\sigma}^\infty$. Then provided that the energy $E( F_{\sigma})$ is bounded from above by some constant $C$, it will be the case that  $E(F_{G_n(\sigma)}^\infty)$ is bounded from above by 
$ C +\delta(n)$ for $\delta(n)$ tending to zero as $n$ tends to infinity, by the fact that the $\Phi_n$ are smoothly converging to isometries on compact sets. Here we also use the fact that the image of any of the harmonic maps that we are considering has diameter bounded above by a constant $C(\M)$ depending only on $\M$: in \cite{Min93}[Lemma 3.1], it is shown that each harmonic map $S \rightarrow \M$ has a bound on the diameter of its image depending only on the injectivity radius lower bound for $\M$ and the topological type of $\M$.

We claim that for any minimizing sequence $\sigma_n$ for the harmonic maps energy functional on {\TS}for $\M$, the sequence $G_n(\sigma_n)$ stays in a compact subset of $\text{Teich}(S)$.  To see this, assume for contradiction that this fails, and fix a collection $\mathcal{F}$ of filling simple closed curves in $S$. \iffalse Recall that we have fixed identifications of $\M$ and $\M_\infty$ with $S \times \mathbb{R}$.   By the implicit function theorem and smooth convergence of the $(M,p_n)$ to $M_{\infty}$, there are harmonic maps $F_{G_n \sigma_n}^{\infty}:S \rightarrow M_{\infty}$ at a small distance from the harmonic maps $F_{\sigma_n}$ postcomposed with $\Phi_n$ cite{sampsonpaper}. \fi  
By smooth convergence of $M$ with basepoints on the images of the $F_{\sigma_n}$ to $M_{\infty}$ and the diameter upper bound from the previous paragraph, we know that the images of the $F_{G_n \sigma_n}^{\infty}$ are contained in a bounded subset $B$ of $M_{\infty}$. 

Suppose for contradiction that $G_n \sigma_n$ leaves every compact subset of $\text{Teich}(S)$.   We can find a constant $C$ and a sequence of mapping classes $T_n$ so that the length of every element of $T_n \mathcal{F}$ in $G_n \sigma_n$ is bounded by $C$.  By the fact that taking the lengths of elements of $\mathcal{F}$ defines a proper function from Teichmuller space to a Euclidean space of dimension equal to the number of curves in $\mathcal{F}$ (this follows, for example, from \cite{bonahon1988geometry}[Proposition 3]), it will then be the case that for at least one of the homotopy classes in $M_{\infty}$ corresponding to the $T_n \mathcal{F}$, any representative contained in $B$ has length at least $K(n)$, where $K(n)$ tends to infinity as $n \to \infty$.   Up to passing to a subsequence we can assume that this is the case for some fixed $\gamma \in \mathcal{F}$.  Since the $\Phi_n \circ F_{\sigma_n}$ are all contained in $B$ and because the infimal length of a representative of $T_n \gamma$ in $B$ tends to infinity as $n \to \infty$, it follows from (\cite{minsky1992harmonicsurfaces}[Proposition 3.1], \cite{minsky1992harmonic}[Section 5]) that the energies of the $\Phi_n \circ F_{\sigma_n}$ tend to infinity, but this is a contradiction.

Thus using the fact that $\Phi_n$ is converging to an isometry on compact sets, we see that  $G_n(\sigma_n)$ is a minimizing sequence for the energy functional $E$ on $\text{Teich}(S)$ for $\M_\infty$, defined using the homotopy equivalence $f_{\infty}:S \rightarrow \M_{\infty}$ in the same way for $\M_{\infty}$ as for $\M$. Since the $G_n(\sigma_n)$ stay in a compact subset of $\text{Teich}(S)$, we can pass to a subsequential limit to obtain an absolute minimum for $E$ on $\M_\infty$. The image of the {\hmp} corresponding to the absolute minimum is then an embedded essential {\ms} in $\M_{\infty}$ \cite{SY79}.  Note that this {\ms} is not just stable but a local minimizer for area (see e.g. \cite{CalegariMinimalSurfaces}[4.4.1].) This fact will be important for the next step, see Remark \ref{rem:Anderson}. In fact, one can check that this minimal surface is a global minimizer for area in its homotopy class in $\M_{\infty}$, although that won't matter for the argument that follows. 

We now apply \Cref{lem:Anderson} to $\M_\infty$.  We just need to check that if we are in the second alternative of \Cref{lem:Anderson}, then $\M$ \iffalse as in the previous subsection\fi has infinitely many minimal surfaces.  In this case there are surfaces  $S_\epsilon$ and $S_{-\epsilon}$  that bound a product region homeomorphic to $S \times (-\epsilon,\epsilon)$  in $\M_{\infty}$ containing $S$, and whose mean-curvature vectors point towards $S$ at every point.  In this case, for large $n$ we have that $\Phi_n(S_{\pm \epsilon})$ bound a product region in $\M$ with mean-convex boundary. Each such product region contains a locally area-minimizing minimal surface by \cite{meeks1982embedded}, which gives infinitely many minimal surfaces in $\M$ because the regions bounded by the $\Phi_n(S_{\pm \epsilon})$ will leave every compact set as $n$ tends to infinity.  Using local min-max theory, we can then produce infinitely many saddle point minimal surfaces as well \cite{KLS19}[Section 6]. 
\end{pf}

\begin{rem} [Gromov's dichotomy vs. our result] Our alternative result, when combined with Gromov's dichotomy~\cite{gromov2014plateau}, gives an interesting relation. In particular, Gromov's dichotomy states that a complete non-compact manifold either contains a finite area minimal surface or there exists a strictly mean convex foliation of every compact subdomain. Assuming no rank-1 cusps in $\M$, \Cref{thm:product-dichotomy} implies that if there is a doubly degenerate manifold $\M$ with bounded geometry which admits a strictly mean-convex foliation (hence no closed minimal surface by Gromov's dichotomy), then there must be another such manifold $\M$ with a minimal foliation.     
\end{rem}

\subsection{Minimal Surfaces in Quasi-Fuchsian Manifolds} \label{sec:QFcount} \
The existence of closed (incompressible) {\ms}s in any {\qfm} was established in the 1970s (\cite{SY79, SU82}). Moreover, there exist {\qfm}s which admit arbitrarily large numbers of stable incompressible {\ms}s(\cite{HW15, huang19complexlength}) (there are always at most finitely many \cite{anderson1983cmh}.) In this subsection, we apply \Cref{thm:product-dichotomy} to obtain results on the growth rate of the number of disjoint {\ms}s in a {\qfm} as a function of the volume of its convex core.

First we prove the following conditional result:
\begin{thm} \label{thm:conditional}
Let $M =\Sigma \times \mathbb{R}$ be an $\epsilon$-thick {\qfm} (without rank-1 cusps), and let $V$ be the volume of its convex core. Then provided that no {\dd} {\htm} admits a minimal foliation, there is a constant $C=C(\epsilon)$ such that any collection of disjoint stable minimal surfaces in $M$ homotopic to $\Sigma$ has at most $C V$ elements.
\end{thm}

\begin{pf} 
We argue by contradiction. Suppose not, and take a sequence of {\qfm}s $M_n$ as above with at least $nV_n$ disjoint stable {\ms}s, where $V_n$ is the volume of the convex core of $M_n$.

Since the injectivity radius of $M_n$ is at least $\epsilon$, the {\sff} of any stable {\ms} in $M_n$ is bounded from above in norm by some constant. For each $M_n$, we can find $k=k(n)$ stable minimal surfaces $\Sigma_n^1,...,\Sigma_n^k$ so that the compact region any two of them bound has volume at most $\delta(n)$, for $\delta(n)$ tending to zero as $n$ tends to infinity and $k$ tending to infinity as $n$ tends to infinity. Because the second fundamental forms of the $\Sigma_n^i$ and $\Sigma_n^j$ are uniformly bounded, this means that $\Sigma_n^i$ and $\Sigma_n^j$ are at a Hausdorff distance from each other tending to zero as $n$ tends to infinity.

Passing to a pointed limit of the $M_n$ with basepoints on the $\Sigma_n^1$, we obtain a {\dd} manifold with bounded geometry in the limit $\M$ containing a minimal surface $\Sigma$ to which the $\Sigma_n^i$ converge. Since the Hausdorff distance between any two $\Sigma_n^j$ goes to $0$ as $n \to \infty$, the limit is the same for any choice of $i=i(n)$. As in the proof of Lemma \ref{lem:Anderson} we can apply the real analytic implicit function theorem for Banach manifolds to the functional $F(w,t) = H_w(exp_x(w(x)\nu(x))) - t\phi(x)$ \iffalse (where $H$ denotes the mean curvature vector, $\nu(x)$ the normal vector to the point $x\in\Sigma$ and $\phi$ is the Jacobi field of the second variation, standarized to be positive and with $L^2$-norm equal to $1$)\fi to obtain a local foliation around $\Sigma$ by closed surfaces of mean curvature proportional to $\phi$. (Recall that $H_w(exp_x(w(x_0)\nu(x_0)))$ is the real valued function on $\Sigma$ defined by taking the inner product of the mean curvature vector of the graphical surface parametrized by $x \mapsto exp_x(w(x)\nu(x))$ at $exp_{x_0}(w(x_0)\nu(x_0))$, with the unit normal vector to that surface at $exp_{x_0}(w(x_0)\nu(x_0))$ pointing towards the same end of $M$ as the normal vector field $\nu$ to $\Sigma$.)    

By smooth convergence of the $\Sigma_n^i$ to $\Sigma$, it follows that we also have foliations around each $\Sigma^i_n$ by closed surfaces of mean curvature a multiple of $\phi^i_n$, the first eigenfunction of the second variation operator of $\Sigma^i_n$ (positive with $L^2$ norm equal to $1$), such that the neighborhoods where the implicit function theorem applies are uniform. This is because the linearized operators of the sequence are close to the linearized operator in the limit, and it follows from the standard contraction mapping proof of the inverse function theorem that the size of the neighborhood to which the inverse function theorem applies depends on the size of the inverse of its linearization.  Let $F^i_n \in C^{\infty}(\Sigma_n^1)$ be the function of which $\Sigma_{n}^i$ is the image under the normal exponential map for $\Sigma_n^1$. Then for $n$ sufficiently large each $\Sigma^i_n$ is a leaf of the foliation associated to $\Sigma^1_n$, since $(F^i_n,0)$ solves $H_w(exp_x(w(x)\nu(x))) - t\phi^i_n(x)=0$ while belonging to the neighborhood where implicit function theorem applies.  

Let $\gamma_n$ (resp. $\gamma$) be the one-dimensional analytic submanifold of $C^{2,\alpha}(\Sigma_n^i)$ (resp. $C^{2,\alpha}(\Sigma)$), homeomorphic to an interval, given by the implicit function theorem.  Define the functions $A_n$ and $A$ on respectively $\gamma_n$ and $\gamma$ by for each point on $\gamma_n$ or $\gamma$ taking the area of the associated surface in respectively $M_n$ or $M$.  Then we have that $\gamma_n$ smoothly converges to $\gamma$ and that $A_n$ smoothly converges to $A$. Here we view the $\gamma_n$ as submanifolds of $C^{2,\alpha}(\Sigma)$ so that this makes sense, by choosing analytic diffeomorphisms between $\Sigma_n^i$ and $\Sigma$ that smoothly converge to the identity and using the corresponding identifications between the $C^{2,\alpha}(\Sigma_n^i)$ and $C^{2,\alpha}(\Sigma)$.

Since $\Sigma^1_n,\ldots,\Sigma^k_n$ correspond to critical points of $A_n$, by iterated applications of Rolle's theorem we have that for any $1\leq i\leq k$ there exists a leaf between $\Sigma^1_n$ and $\Sigma^k_n$ so that the $i$-th jet of $A_n$ vanishes. Moreover, as we have smooth convergence of all $\Sigma^i_n\xrightarrow[n\rightarrow+\infty]{} \Sigma$ for all $i$ \iffalse independent of $i$\fi , it follows that the leaves for which the $i$-th jet of $A_n$ vanishes accumulate at $\Sigma$, and consequently the $i$-th jet of $A$ vanishes at $\Sigma$ for any $i\geq 1$. Since $A$ is analytic it follows that we have then a local foliation around $\Sigma$ by compact minimal leaves, and such a foliation can be extended to be global as in the proof of Theorem \ref{thm:product-dichotomy}, which leads to a contradiction. 
\end{pf}

We also have the following unconditional result, which follows from an argument very similar to the proof of the last theorem. One uses the fact that hyperbolic 3-manifolds with short geodesics as in the statement of the next theorem cannot admit minimal foliations by \cite{huang19complexlength}, \cite{hass2015minimalfibrationshyperbolic3manifolds} (the latter dealt with closed hyperbolic 3-manifolds that fiber over the circle, but the same arguments apply in our setting.) Thus a geometric limit of quasi-Fuchsian hyperbolic 3-manifolds as in the next theorem cannot admit a minimal foliation.  

\begin{cor} \label{cor:nonconditional}
Suppose $M$ is a {\qfm} satisfying: 

\begin{enumerate}
\item $M$ is $\epsilon$-thick (no rank-1 cusp),
\item There is some $R>0$ so that each ball of radius $R$ in the convex core of $M$ contains a geodesic whose complex length satisfies the conditions in the main theorem of \cite{huang19complexlength}. 
\end{enumerate}
Then there is some constant $C=C(\epsilon,R)$ so that any collection of disjoint stable minimal surfaces in $M$ homotopic to the fiber has at most $C \cdot V$ elements, for $V$ the volume of the convex core of $M$.  
\end{cor}

While the above condition seems restrictive, we give next an explicit construction of {\qfm}s as in the corollary.

\subsection{Quasi-Fuchsian Manifolds with Short Geodesics} \label{sec:shortgeodesics} 

In this part, we construct examples of {\qfm}s satisfying the hypotheses of Corollary \ref{cor:nonconditional}.  

\begin{thm} There are quasi-Fuchsian manifolds satisfying the conditions of \Cref{cor:nonconditional}.
\end{thm}

\begin{proof} Choose an element of the mapping class group $\Phi$ that is the monodromy for one of the fibered hyperbolic 3-manifolds constructed by Huang-Wang \cite{huang19complexlength} or Hass \cite{hass2015minimalfibrationshyperbolic3manifolds} that does not admit foliations by minimal surfaces. Then a sequence of quasi-Fuchsian manifolds $M_n$ obtained by applying $\Phi^n$ and $\Phi^{-n}$ to the conformal infinities of some fixed quasi-Fuchsian manifold will give a sequence of {\qfm}s satisfying the assumptions of the previous corollary. \rev{For an appropriate choice of basepoints the sequence $M_n$ converges algebraically to a doubly degenerate $\mathbb{Z}$-periodic limit that is a $\mathbb{Z}$-cover of the hyperbolic 3-manifolds that fiber over the circle constructed by Hass and Huang-Wang, and moreover such convergence is strong (see for instance \cite[Theorem 4.1, Theorem 7.2]{thurstondoublelimit}).}

 To see that all but finitely many of the $M_n$ satisfy the hypotheses of the corollary, assume they did not, and take a sequence of points $p_n \in M_n$ contained in the convex core so that the ball of radius $R_n$ centered at $p_n$ contains no short loop satisfying condition (2) in Corollary \ref{cor:nonconditional} for some $\delta>0$ and $R_n$ tending to infinity as $n \to \infty$.

Then passing to a pointed limit of the $M_n$ we obtain a limit hyperbolic 3-manifold $\M$ which is homeomorphic to $\Sigma \times \mathbb{R}$ and arguments similar to Section \ref{sec:Schottky} show that it must have either one or two degenerate ends, depending on whether the $p_n$ stay at bounded distance from the boundary of the convex core. First assume the pointed limit is doubly degenerate. Then $\M$ must be a $\mathcal{Z}$-cover of the fibered hyperbolic 3-manifolds constructed in \cite{huang19complexlength}, which contain loops satisfying (2) as above for a contradiction. Second, if $\M$ has only one degenerate end, then this end must be asymptotic to a $\mathcal{Z}$-cover of the fibered hyperbolic 3-manifolds constructed in \cite{huang19complexlength}, and so far enough into the end there are closed curves satisfying (2) which gives a contradiction in this case too.  
\end{proof}

\begin{cor}\label{cor:dichotomyQF}
    A quasi-Fuchsian manifold (without rank-1 cusps) satisfies one of the following statements:
    \begin{enumerate}
        \item\label{QF:unique} Contains a unique closed local area minimizer isotopic to the fiber.
        \item\label{QF:infinite} Contains infinitely many saddle point minimal surfaces.
    \end{enumerate}
\end{cor}
\begin{proof}
    Assume that $M=\Sigma\times\mathbb{R}$ is a quasi-Fuchsian manifold with more than one closed local area minimizer \rev{isotopic to the fiber}. Without loss of generality, we can assume these local area minimizers $\Sigma, \Sigma'$ are disjoint. Indeed, if they intersect, \rev{for each unbounded component of their complement we can consider its intersection with the convex core and minimize area within this region for surfaces isotopic to the fiber, as each of these regions has boundary that is piecewise mean convex with dihedral angle less than $\pi$. This produces a disjoint pair of local area minimizers isotopic to the fiber, as desired.} By Corollary \ref{cor:zero-infinity} it is sufficient to prove that there exists a saddle point minimal surface. Let $\Omega$ be the compact region bounded by $\Sigma, \Sigma'$ and take $\Sigma=\Sigma_1,\ldots,\Sigma_n=\Sigma'$ a maximal collection of stable minimal surfaces isotopic to $\Sigma\times\lbrace 0\rbrace$ in the region $\Omega$, so that each consecutive pair bounds a cylinder disjoint from any of the other surfaces. Without loss of generality assume that all surfaces in this maximal collection are not saddle. Then at least one consecutive pair bounds a cylinder $C$ that satisfies the conditions of \cite[Proposition 7.6]{GLP21}, \rev{namely that there are no stable minimal surfaces isotopic to the fiber in the interior of $C$ and there is a foliation near $\partial C$ with non-zero mean curvature vector pointing towards the boundary. The first condition follows by the maximality of the collection $\Sigma=\Sigma_1,\ldots,\Sigma_n=\Sigma'$ for any consecutive pair. As we have foliations near $\Sigma, \Sigma'$ with non-zero mean curvature vector pointing towards each of them, we must have at least one consecutive pair satisfying the second condition of \cite[Proposition 7.6]{GLP21}. Hence as desired it follows the existence of a saddle point minimal surface in any such cylinder $C$.}
\end{proof}

\begin{rem}
    There are examples of either case in Corollary \ref{cor:dichotomyQF}. As mentioned before, almost Fuchsian manifolds (\cite{uhlenbeck1983closed, HW13}) satisfy condition (\ref{QF:unique}). For examples satisfying (\ref{QF:infinite}) one can take the manifolds from \cite{huang19complexlength}. Alternatively,  it is known that there exist quasi-Fuchsian manifolds with at least one degenerate stable minimal surface homotopic to the fiber (see for instance \cite[Theorem 1.6]{GLP21} where that is shown to occur for outermost minimal surfaces). By work of Uhlenbeck (\cite[Theorem 4.4, Corollary 5.2]{uhlenbeck1983closed}) we can deform such quasi-Fuchsian manifold to quasi-Fuchsian manifold with a strictly unstable minimal surface homotopic to the fiber. Such quasi-Fuchsian manifolds satisfy (\ref{QF:infinite}) by Corollary \ref{cor:zero-infinity}.
\end{rem}

\section{Schottky Manifolds} \label{sec:Schottky}

As previously noted, it remains an open question whether a Schottky manifold contains a closed, embedded minimal surface. In this section, we present the first examples of such manifolds.

We emphasize that by a \emph{Schottky manifold} we refer to a handlebody with a convex co-compact (in particular complete) hyperbolic metric.

\begin{thm} \label{thm:Schottky-number}
    Let $C$ be a positive integer and $H_g$ the topological handlebody obtained by attaching $g\geq2$ $1$-handles to the unit ball in $\mathbb{R}^3$. Then the set of Schottky manifolds homeomorphic to $H_g$ that contain at least $C$ local minimizers of area isotopic to their boundary is infinite. Moreover, there exists a geometrically infinite complete hyperbolic metric in $H_g$ with infinitely many local minimizers of area that are isotopic to the boundary.
\end{thm}
\begin{pf}
For the construction we follow \cite[Corollary 1.4, Chapter 5]{NamaziThesis}. Namely, for a marked convex co-compact hyperbolic metric on a handlebody $N$ with conformal structure at infinity $\tau$ consider pseudo-Anosov homeomorphism $\varphi$ of $\partial N$ with attracting lamination $\lambda$, a filling lamination in the \textit{Masur domain}, or the set of ending laminations with non-zero intersection with the compressible curves of $N$. This can be accomplished following \cite[Section 6]{ThurstonPA} by taking $\alpha, \beta$ a filling pair of simple closed curves in $\partial N$ so that $\alpha$ is disk-busting and taking $\varphi_n = T^{n}_\alpha\circ T_\beta$, where $T_{\alpha, \beta}$ are the Dehn twists along the respective curve and $n$ is sufficiently large. As $n$ grows the attracting measured foliation converges to the geometric intersection with $\alpha$, hence $\varphi^n$ will be in the Masur domain for large $n$.

Fixing a particular $\varphi=\varphi_n$, it follows from \cite[Corollary 1.4, Chapter 5]{NamaziThesis} that the sequence $N_i$ of marked convex co-compact metrics in $N$ with conformal structure at infinity $\tau_i=\varphi^i(\tau)$ converges algebraically to $N_\infty$, a geometrically infinite metric in $N$ with ending lamination equal to $\lambda$. By Canary's Covering Theorem (\cite{Canary96}) the algebraic limit $N_\infty$ is a finite index subgroup in any geometric limit $N_{\rm geo}$, which implies that $N_\infty$ is the geometric (and hence strong) limit of the sequence $N_i$. Indeed, let $\gamma \in \pi_1(N_{\rm geo})$. Since $[\pi_1(N_{\rm geo}) : \pi_1(N_\infty)]<\infty$ then there exists $k$ such that $\gamma^k \in \pi_1(N_\infty)$. If $\rho_i:\pi_1(N)\rightarrow PSL_2(\mathbb{C})$ are the discrete faithful representations of $N_i$, then there exist $h,g_i\in \pi_1(N)$ so that $\gamma = \lim_{i\rightarrow\infty} \rho_i(g_i)$ and $\gamma^k = \lim_{i\rightarrow\infty} \rho_i(h)$. Hence $id = \lim_{i\rightarrow\infty} \rho_i(g_i^kh^{-1})$, so by the discreteness of $\rho_i$ it follows that $g_i^k = h$ for $i$ sufficiently large. Indeed, if by way of contradiction we have $g_i^kh^{-1}\neq id$ for infinitely many indices, we can take each $A_k=g_i^kh^{-1}$ as one of the generators of a non-elementary 2 generator group (by choosing any other group element \rev{$B_k$} with different fixed points at infinity for the other generator),  which would violate Jorgensen's inequality. (Recall that Jorgenson's inequality states that for $A$ and $B$ generators of a nonelementary Kleinian group, 
\[
|\text{Tr}(A)^2 - 4 | + |\text{Tr}(ABA^{-1}B^{-1} - 2| \geq 1, 
\]
\rev{which cannot hold for a pair of convergent generators $A_k,B_k$ so that $A_k\rightarrow id$, as the left-hand side would converge to $0$}.)

\noindent In particular, the sequence $(g_i)$ is constant for $i$ sufficiently large, which implies $\gamma\in\pi_1(N_\infty)$.

By the Ending Lamination Theorem (\cite{BCM12EL,BowditchEL}), we have that in the end of $N_\infty$ the translation defined by $\varphi$ is a quasi-isometry. Hence $\varphi$ gives a $\mathbb{Z}$-action on the limit surface group obtained by exiting the end of $N_\infty$. Such surface group will be then the cyclic cover of the hyperbolic mapping torus $M_\varphi$ induced by $\varphi$. Hence, taking $\varphi$ so that $M_\varphi$ does not foliate by minimal surfaces (taking it for instance so that $M_\varphi$ is a sufficiently high Dehn-filling of a non-compact, finite volume hyperbolic manifold following \cite{hass2015minimalfibrationshyperbolic3manifolds}, \cite[Theorem 2.5]{farre2022minimal}, which can be accomplished by taking $n$ sufficiently large in $\varphi=\varphi_n$) then an area minimizer $\Sigma$ among surfaces isotopic to the fiber of $M_\varphi$ has a strictly mean convex foliation around it by \Cref{lem:Anderson}. 
From here the argument follows by \cite[Corollary 5.2]{farre2022minimal}. By geometric convergence, we have infinitely many regions exiting the end of $N_\infty$ that are diffeomorphic to $\partial N\times[0,1]$ with boundary isotopic to $\partial N$ and with mean convex boundary. Thus, we can find area minimizers within these regions to produce minimal surfaces in $N_\infty$ isotopic to $\partial N$.

For any $C>0$ we can consider a compact set $K\subseteq N_\infty$ that contains at least $C$ of the regions with mean convex boundary. Applying again geometric convergence, we have that for sufficiently large $i$ the manifold $N_i$ has at least $C$ regions that are diffeomorphic to $\partial N\times[0,1]$ with boundary isotopic to $\partial N$ and with mean convex boundary. Then by the same argument as before, such $N_i$'s will have at least $C$ local minimizers of area that are isotopic to the boundary.
\end{pf}

Hence, by a direct application of Lemma \ref{lem:dichotomy2}~(\cite[Theorem 1.2]{song2023dichotomy}), we establish the remainder of \Cref{thm:intro-Schottky}, namely that the Schottky manifold (and more generally, bounded geometry handlebodies with at least two boundary isotopic local area minimizers) we constructed with a compact local area minimizer must have then infinitely many saddle point minimal surfaces. 

\begin{cor} \label{cor:Schottky-infty} For any $3$-dimensional handlebody with at least two $1$-handles there exists a complete hyperbolic metric of bounded geometry containing infinitely many closed, embedded, saddle point minimal surfaces. Moreover such hyperbolic manifolds can be taken to be Schottky manifolds.
\end{cor}
\begin{proof}    By using for instance the examples from~\Cref{thm:Schottky-number}, take any bounded geometry handlebody with two disjoint local area minimizers isotopic to the boundary. Applying the reasoning of \cite[Proposition 7.6]{GLP21} as in Corollary \ref{cor:dichotomyQF}, by the local min-max theorem \cite[Section 6]{KLS19} we have a saddle point minimal surface in between. Hence the result follows once more by applying~\Cref{cor:zero-infinity}.
\end{proof}

\begin{rem} [Schottky Manifolds with no minimal surfaces] \label{rem:schottky-no} We note that there are trivial examples of Schottky manifolds with no embedded, complete, finite area minimal surfaces, e.g. $\BH^3$, and Fuchsian Schottky Manifolds~\cite[Lemma 4]{coskunuzer2020minimal}. Note that the nonexistence of non-compact, finite area minimal surfaces in these manifolds follows from~\Cref{lem:thickness}.
\end{rem}

\section{Hyperbolic $3$-Manifolds with Rank One Cusps} \label{sec:rank1cusps}

In~\cite{coskunuzer2020minimal}, the existence of closed minimal surfaces in hyperbolic 3-manifolds was established, excluding the case of rank-1 cusps (\Cref{lem:existence}). In this section, we study this case.

In the following, relatively incompressible refers to the incompressibility of a map $f:(S,\partial S)\rightarrow (M,P)$ where $S$ is a surface with boundary and $P$ is the union of the annuli corresponding to rank-1 cusps linking geometric ends together in the topological ends of a tame hyperbolic $3$-manifold.

\begin{defn}
    We say a geometric end $\E$ of a tame hyperbolic manifold $\M$ is \emph{essential after compressions} if after performing some disk compressions (possibly none) on $\E$ we produce at least one connected component that is relatively incompressible but it is not homotopic to any geometric end other than possibly $\E$.
\end{defn}

As an example, if $\M$ is neither a product nor a compression body then any end is essential after compressions. Moreover, if $\M$ is not a product then any relatively incompressible end is essential after compressions. Theorem \ref{thm:intro-rank1} is a consequence of the next theorem.  

\begin{thm}\label{thm:existence_1+geofiniteend}
    Let $\M = \mathbb{H}^3/\Gamma$ be an infinite volume complete hyperbolic $3$-manifold with rank-1 cusps that is homeomorphic to the interior of a compact manifold. If $\M$ contains a geometrically finite geometric end $\E$ that is essential after compressions, then $\M$ contains a finite area, complete, embedded, stable minimal surface $\s$.
\end{thm}

\begin{rem}
 In the case of quasi-Fuchsian hyperbolic manifolds homeomorphic to the product of a cusped surface with $\mathbb{R}$, the paper \cite{HLS23} proved existence of complete embedded finite area minimal surface. Such manifolds satisfy the assumptions of the first sentence of the theorem but not the second.  
\end{rem}
\begin{pf}
    Let $\Sigma$ be the boundary of the convex core of $\M$ in $\E$. The main idea is to find an area minimizer starting with $\Sigma$ and using $\Sigma$ as a barrier, where the topological assumption is to prevent minimizing sequences from escaping to infinity. As the results of \cite{meeks1982embedded} require a closed surface, we will modify and double the metric of $\M$ to obtain a sequence of manifolds $\widehat{\M}_n$ with area minimizers $\s_n$ and so produce the desired minimal surface $\s$ as a limit of $\s_n\in \widehat{\M}_n$.
    
    Let $\{\U_i\}_{i=1}^m$ be the rank-1 cusp neighborhoods for all of the cusp ends of $\Sigma$. Without loss of generality we take $\{\U_i\}_{i=1}^m$ so that $\{\Sigma\cap\U_i\}_{i=1}^m$ are all totally geodesic components (see \cite[Chapter 4]{canary2006notes}).  We claim that the pleating locus on $\Sigma$ is disjoint from $\cup_{i=1}^m \U_i$.  Indeed, as $\Sigma$ is a finite area pleated surface we can take cusp neighborhoods so that every $2$ dimensional totally geodesic stratum close to the cusp accumulates into it. In the universal cover, each such stratum lifts to a totally geodesic region $R$ with totally geodesic boundary so that the boundary at infinity $\partial_\infty R$ contains a rank-1 parabolic fixed point $p_0$ of $\Gamma$. As $R$ is given by a supporting plane of the convex hull of $\Gamma$ for any of its points, it must be invariant by the rank-1  parabolic subgroup of $\Gamma$ that fixes $p_0$. Moreover any such region $R$ must belong to the same invariant totally geodesic plane, from which the claim follows.

    For each $\U_i$ take a parametrization by horocylinders $\mathbb{S}^1\times \mathbb{R} \times\mathbb{R}^+$ (where every slice $S^1\times \mathbb{R} \times \lbrace t\rbrace$ is the quotient of an horoball by the corresponding rank-1 parabolic subgroup of $\Gamma$). Let $\U^n_i \subset \U_i$ given by $\mathbb{S}^1\times \mathbb{R} \times]n,+\infty[$ and let $\M_n : = \M \setminus \left( \cup_{i=1}^m \U^n_i \right)$. Let $g$ be the hyperbolic metric in $\M$ and consider a conformal change of metric $\Tilde{g}_n = e^{2u_n}g$ in $\M_n$ with $u_n$ satisfying the following properties
    \begin{enumerate}
        \item $u_n$ is a smooth function supported in $\mathbb{S}^1 \times\mathbb{R} \times[n-1,n+1]$.
        \item In $\mathbb{S}^1 \times\mathbb{R} \times[n-1,n]$ the level sets of $u_n$ are given by horocylinders $S^1\times \mathbb{R} \times \lbrace t\rbrace$.
        \item $\partial_t u_n \leq 1$ for $t
        \in [n-1,n]$, with equality at $\mathbb{S}^1\times \mathbb{R} \times \lbrace n\rbrace$.
    \end{enumerate}
    It is straightforward to verify that such conformal changes exist.
    
    The mean curvature of a surface under a conformal change of metric is given by
    \[
    H_{\tilde{g}} = e^{-u}\left( H_g - D_nu\right),
    \]
    where $n$ is the normal of a surface at a given point. Since by construction we have that $\Sigma$ is orthogonal to the level sets of $u_n$, it follows that $\Sigma$ is mean convex with respect to $\tilde{g}_n$.

    By the third condition on $u_n$ we have that $\lbrace S^1\times \mathbb{R} \times \lbrace t\rbrace \rbrace_{t\in[n-1,n]}$ is a mean concave foliation with $S^1\times \mathbb{R} \times \lbrace n\rbrace$ minimal (i.e., the mean curvature vectors of the $\lbrace S^1\times \mathbb{R} \times \lbrace t\rbrace \rbrace_{t\in[n-1,n)}$ point towards $S^1\times \mathbb{R} \times \lbrace n\rbrace$.) Moreover, as one can check that under our assumptions each horocylinder stays umbilic, it follows that $S^1\times \mathbb{R} \times \lbrace n\rbrace$ is in fact totally geodesic in $\Tilde{g}_n$.
    
    Let $\widehat{\M}_n$ be the double of $\M_n$ through its totally geodesic boundary in the metric $\Tilde{g}_n$. Observe that $\widehat{\M}_n$ contains two isometric copies of $\M_{n-1}$. Denote by $\widehat{\Sigma}_n$ the mean-convex surface in $\widehat{\M}_n$ given by the double of $\Sigma \cap \M_n$. As $\E$ is essential after compressions and compressions of $\widehat{\Sigma}_n$ can be done on each half of the double, after finitely many compressions of $\widehat{\Sigma}_n$ (possibly none) we have an incompressible closed surface $\Sigma_n$ in $\widehat{\M}_n$ that is not homotopic to any other end of $\widehat{\M}_n$ other than possibly the end containing $\widehat{\Sigma}_n$. As $\widehat{\Sigma}_n$ is mean-convex, by \cite{meeks1982embedded} we have an area minimizer $\s_n$ obtained from the homotopy class of $\Sigma_n$. As the complement of the copies of $\M_0\subseteq \M_{n-1}$ in $\M_n$ has a mean-convex foliation, it follows that $\s_n$ must intersect one of these copies of $\M_0\subseteq \M_{n-1}$. As $\s_n\cap \M_{n-1}$ is non-peripheral with respect to all ends except possibly the initial geometrically finite end, \iffalse maybe except $\Sigma \cap \M_{n-1}$ \fi it follows that there is a compact set $K\subset \M$ that all $\s_n$ intersect. Hence by making $n\rightarrow+\infty$ and taking a subsequence if necessary, we obtain a finite area minimal surface $\s\subset \M$ as the limit of $\s_n\cap \M_{n-1}$.
\end{pf}

\section{Final Remarks} \label{sec:remarks}

\subsection{Summary of Existence Results} 

In this paper, we study the existence and nonexistence of finite area, embedded minimal surfaces in infinite-volume hyperbolic $3$-manifolds.  After our results, here is a brief summary of the known existence results:

\smallskip

\noindent $\diamond$ Let $\M$ be a complete {\htm} with \textbf{finite volume}:
\begin{itemize}[left=15pt]
\item {\bf $\M$ is closed:} There exists infinitely many closed, embedded minimal surfaces in $\M$~\cite{KM12, song2018existence}.
\item {\bf $\M$ is non-compact:} There exist closed minimal surfaces in $\M$~\cite{collin2017minimal,huang2017closed,chambers2020existence,song2023dichotomy}.
\end{itemize}

\smallskip 

\noindent $\diamond$ Let $\M$ be a hyperbolic $3$-manifold with \textbf{infinite volume without rank-1 cusps}:
\begin{itemize}[left=15pt]
\item {\bf $\M$ is geometrically finite:} There exists a closed, embedded, incompressible minimal surface in $\M$~\cite{SY79,SU82}.
\item {\bf $\M$ is geometrically infinite:} If $\M$ is topologically not a product or a handlebody, then there exists a closed embedded minimal surface in $\M$~\cite{coskunuzer2020minimal}.
\item {\bf $\M$ is geometrically infinite, product, bounded geometry:} Either any such $\M$ contains a closed embedded minimal surface, or there exists such a manifold with minimal foliation~(\Cref{thm:product-dichotomy}).
\item {\bf $\M$ is Schottky (or handlebody):} Schottky manifolds can contain infinitely many closed, embedded minimal surfaces, or none with finite area~(\Cref{cor:Schottky-infty}, \Cref{rem:schottky-no}).
\end{itemize}

\smallskip

\noindent $\diamond$ Let $\M$ be a hyperbolic $3$-manifold with \textbf{infinite volume with rank-1 cusps}:
\begin{itemize}[left=15pt]
\item {\bf $\M$ contains a geometrically finite geometric end:} If $\M$ is topologically not a handlebody or product, then there exists a finite area, complete, embedded, minimal surface in $\M$ (\Cref{thm:existence_1+geofiniteend}, and when $\M$ is {\qf}, see~\cite{HLS23}). 
% \item {\bf $\M$ is geometrically infinite:} If $\M$ is not a handlebody nor a product, and contains a convex barrier surface in one of the ends, then there exists a finite area, complete,  embedded, minimal surface in $\M$~(\Cref{thm:existence-rank1}).

\end{itemize}

\subsection{Open Questions} \label{sec:open_questions} 

In this part, we discuss some open questions and potential approaches. For {\dd}manifolds with bounded geometry, we proved that either any such manifold contains a closed minimal surface, or if not there must be a doubly degenerate manifold which has a minimal foliation. Notice that this result is closely related with another important question posed by Thurston decades ago, also known as Anderson's Conjecture (for a negative answer) \cite{anderson1983cmh}:

\begin{question} \cite{rubinstein2007problems} \label{question:foliation} Is there a closed {\htm} that contains a (local) foliation by minimal surfaces?    
\end{question}

This is a decades-old famous question in the domain which relates to our alternative result through Virtual Fibering Theorem due to Agol~\cite{agol2013virtual}. With this result, \textit{any closed {\htm} with minimal foliation has a finite cover which fibers over a circle.} Hence, the infinite cover of such a fibered hyperbolic manifold will give a doubly degenerate manifold with bounded geometry as in the statement of \iffalse stated in our dichotomy~(\fi Theorem \ref{thm:product-dichotomy}. A negative answer to \Cref{question:foliation} does not rule out the second case in the alternative, as there might be foliations in doubly degenerate manifolds, which do not descend to closed manifolds.

Although our dichotomy result sheds some light on the bounded geometry case, the unbounded geometry case is wide open.

\begin{question} \label{question:unbounded} Let $\M$ be a doubly degenerate hyperbolic $3$-manifold with unbounded geometry. Is there a finite area, complete, embedded minimal surface in $\M$?    
\end{question}

Our proof of \Cref{thm:product-dichotomy} used the bounded geometry assumption in an essential way. If one wants to adapt the arguments to apply to the unbounded geometry case, allowing arbitrarily short geodesics, a more complicated analysis is needed to find a minimizer for the energy functional. This involves serious difficulties first in the harmonic maps part of the argument, where one would need to consider harmonic maps from possibly cusped surfaces into targets without bounded geometry.  There are also complications in applying the theory of convergence for Kleinian groups, which is much more involved if one drops the bounded geometry assumption.  Even so, we believe that one could obtain similar results in this setting, and that it is an interesting topic for future work.

The next question concerns the existence problem for minimal surfaces in general hyperbolic 3-manifolds. All known examples of hyperbolic $3$-manifolds with no finite area, embedded minimal surfaces are geometrically finite (\Cref{rem:schottky-no}). Furthermore, our dichotomy (\Cref{cor:zero-infinity}) shows that the existence of one saddle point minimal surface implies the existence of infinitely many finite area, embedded minimal surfaces. Hence, we pose the following question.

%\iffalse is about zero-infinity dichotomy for  non-cusped  {\htm}s.\fi 

\begin{question} Let $\M$ be a geometrically infinite hyperbolic $3$-manifold with no cusps. Then, is it true that \iffalse if \fi $\M$ contains infinitely many finite area, complete, embedded, minimal surfaces?
\end{question}

For Schottky manifolds, following our zero-infinity dichotomy result and positive examples, a natural question arises:

\begin{question} Characterize the Schottky manifolds that have no closed embedded minimal surface. 
\end{question}

We also ask whether Theorem \ref{thm:conditional} and Corollary \ref{cor:nonconditional} on upper bounds for the number of disjoint stable minimal surfaces in the convex core of a quasi-Fuchsian manifold can be strengthened.

\begin{question}
Is there a computable upper bound for the number of (not necessarily disjoint) stable embedded minimal surfaces in a quasi-Fuchsian manifold, in terms of the volume of its convex core? 
\end{question}

In the setting of hyperbolic 3–manifolds with rank-1 cusps, although we addressed an important special case, the general problem remains wide open.

\begin{question} Let $\M$ be a complete hyperbolic $3$-manifold with rank-1 cusps. Does $\M$ contain a finite area, complete, embedded minimal surface?
\end{question}

A promising strategy is to generalize the shrinkwrapping techniques from \cite{coskunuzer2020minimal} to this cusp setting. However, the absence of a uniform positive lower bound on the injectivity radius makes several steps of that construction highly delicate.  %The second question regarding manifolds with rank-1 cusps is seemingly straightforward, but again wide open.

% \begin{question} Let $\M$ be a geometrically infinite hyperbolic $3$-manifold with $\M=\F\times \R$ where $\F$ is a punctured surface. Does $\M$ contains a finite area, complete, embedded, minimal surface?
% \end{question}

\bibliographystyle{alpha}
\bibliography{references}

\end{document}